\documentclass{article}
\usepackage{enumerate}
\usepackage{amsmath, amssymb}
\usepackage{amsthm}
\usepackage{graphicx}

\renewcommand{\Re}{\mathrm{Re}}
\renewcommand{\Im}{\mathrm{Im}}
\renewcommand{\phi}{\varphi}

\newcommand{\quotes}[1]{``#1''}

\newcommand{\same}{\Leftrightarrow}

\theoremstyle{definition}

\newcommand{\defname}{Definition}
\newtheorem{Def}{\defname}
\newcommand{\propname}{Proposition}
\newtheorem{prop}{\propname}
\newcommand{\theoname}{Theorem}
\newtheorem{thm}{\theoname}
\newcommand{\corname}{Corollary}
\newtheorem{Cor}{\corname}
\newcommand{\lemname}{Lemma}
\newtheorem{Lem}{\lemname}

\DeclareMathOperator{\Ker}{Ker}

\newcommand{\set}[1]{\{ #1 \}}
\newcommand{\del}{\partial}
\newcommand{\R}{\mathcal{R}}

\newcommand{\RR}{\mathbf{R}}
\newcommand{\CC}{\mathbf{C}}
\newcommand{\zero}{\mathbf{0}}
\newcommand{\m}{\mathfrak{m}}
\newcommand{\E}{\mathcal{E}}
\newcommand{\Harm}{\mathcal{H}}
\title{Deformations of harmonic function germs}
\author{Yuki Yasuda}
\begin{document}
\maketitle

\begin{abstract} % {{{
We study the classification problem of singularities of function-germs with harmonic leading terms
  of two variables under the right-equivalence.
We observe that the multiple actions of Laplacian appear for the classifications
  of such class of function-germs (\theoname \ \ref{main}).
\end{abstract} % }}}

\section{Introduction} % {{{

%% purpose
In this paper, we study the classification of singularities of function-germs of two variables with
  a non-zero leading term under the right equivalence.

%% Laplacian, Harmonic, f_k, g_k, H_k(R^2)
Recall that a function $h$ on $\RR^n$ is called a harmonic function if it satisfies the Laplace
equation $\Delta h = 0$, where $\Delta = \sum_{j=1}^n \frac{\del^2}{\del x_j^2}$ denotes the
  Laplacian.
In this paper, we study singularities of functions of two variables $x$ and $y$.
Therefore we put $\Delta = \frac{\del^2}{\del x^2} + \frac{\del^2}{\del y^2}$.

We will prove that the multiple actions of Laplacian appear for the classifications
  of function-germs with harmonic leading terms of two variables(\theoname \ \ref{main}).

For example, for all $h_k$ homogeneous harmonic polynomial-germ
  from $(\RR^2,\zero)$ to $(\RR, \zero)$ of degree $k$ and
  $R_{k+1}$ with the condition that two times Laplacian of $R_{k+1}$ vanishes,
  $h_k$ and $h_k + R_{k+1}$ are right equivalent (\corname \ \ref{biharm}).

We proved \theoname \ \ref{main} by some computing in case of
  $\text{(order of the function-germ)} \le 7$ (see \cite{mine}).

Recall that any harmonic function-germ on $(\RR^2, \zero)$ is a real part of
a halomorphic function on $(\CC, 0)$.

In this paper, we use the notation for the harmonic polynomials of special type:
\[
  f_k = \Re (x + iy)^k, \quad
  g_k = \Im (x + iy)^k = \Re (-i(x+iy)^k),
\]
where $i=\sqrt{-1}$ and $\Re$ (resp. $\Im$) means the real (resp. imaginary) part.

Let $\Harm_k(\RR^2)$ denote the vector space of all homogeneous harmonic polynomial-germs
  from $(\RR^2,\zero)$ to $(\RR, \zero)$ of degree $k$.
  It is known that $\Harm_k(\RR^2)$ is spanned by $f_k$ and $g_k$.
  In particular, $\dim_\RR \Harm_k(\RR^2) = 2$.

Also, we recall the right equivalence for the singularities.

%%% right equivalent
We say that $h_j\colon (\RR^m, \zero) \to (\RR^n, \zero) (j=1, 2)$ are right equivalent (written
  $h_1 \sim_{\R} h_2$) if there exists a diffeomorphism-germ $\phi \colon (\RR^m, \zero) \to
  (\RR^m ,\zero)$ such that $h_1\circ \phi = h_2$.

%%% order of Delta == 2
By the main Theorem of this paper(\theoname \ \ref{main}), we obtain following statement:

\begin{Cor} \label{biharm}
Let us $k$ be a natural number more than or equal to 5, non-zero homogeneous harmonic polynomial
  $h_k$ with degree $k$ and
  the function-germ $R$ with order more than $k$ with condition $\Delta^2 R = 0$.
  Then, $h_k$ and $h_k + R$ are right equivalent.
\end{Cor}

%%% k <= 4
Note that if $k\le 4$, then any non-zero harmonic function $h_k$ with degree $k$ is $k$-determined i.e.
  for all $R$ with order more than $k$, $h_k$ and $h_k + R$ are right equivalent
  (see \cite{mine}).

%}}}

\section{Preparations} % {{{

\subsection{Preparations about singularities} % {{{
We recall the basic definitions and  a statement about singularities.

Let $k$ and $n$ be natural numbers in this subsection.

Let $\E_n$ be the set consisting of $C^\infty$ map-germs from $(\RR^n, \zero)$ to $\RR$
  and $\m_n \subset \E_n$ denote the ideal consisting of $C^\infty$ map-germs from $(\RR^n, \zero)$
  to $(\RR, 0)$.

Let $h_1, h_2 \in m_n$.
We call $h_1$ and $h_2$ are $k$-jet equivalent(written $h_1 \sim_{j^k} h_2$) if $h_1 - h_2 \in
\m_n^{k+1}$.

Let $h \in \m_n$.
Then, $J_h$ denotes the Jacobian ideal of $h$ i.e.
  $J_h = \langle \frac{\del h}{\del x_j}| 1 \le j \le n \rangle_{\E_n}$.

We say that a function-germ $h\colon (\RR^k, \zero) \to (\RR, 0)$ is $k$-determined if
  for any $R \in \m_2^{k+1}$, $h$ and $h + R$ are right equivalent.

We recall $k$-determinancy.

\begin{prop}[see \cite{Mother-3}]  \label{judge-determinancy} % m_2^k \subset m_2 J_f + m_2^{k + 1}
Let $h \in \m_n$.
  Suppose for a natural number $k$, $m_n^k \subset m_n J_h + m_n^{k+1}$ holds.

Then $h$ is $k$-determined.
\end{prop}
% }}}

\subsection{Preparations about (poly) harmonic functions } % {{{
First, we will give basic definitions or statements about harmonic functions.

Let $k$ be a natural number and set $r^2 := x^2 + y^2$.

We will prove some properties about $f_k$ and $g_k$.

\begin{prop} \label{harm-prop1}
For all natural number $k$, we have
\[
  f_{k+1} = xf_k - yg_k, \quad
    g_{k+1} = xg_k + yf_k.
\]
\end{prop}

\begin{proof}
Let $z = x + iy$. By taking real part or imaginary part of $z^{k+1} = z \cdot z^k$, we have the
  \propname \ \ref{harm-prop1}.
\end{proof}

\begin{prop} \label{harm-prop2}
Let $s$ and $k$ be two non-zero integers such that $s \le k$.
Then
\begin{align*}
  f_sr^{2(k-s)} &= f_kf_{k-s} + g_kg_{k-s}, \\
  g_sr^{2(k-s)} &= g_kf_{k-s} + f_kg_{k-s}.
\end{align*}
\end{prop}
\begin{proof}
Let $z = x + iy$.
By taking real part or imaginary part of $r^{2(k-s)}z^s = z^k\bar{z}^{k-s}$, we have the
  Proposition.
\end{proof}

\begin{Def}
Let $s$ be a natural number.
Let $\Omega \subset \RR^n$ be a open set and $u\colon \Omega \to \RR$ be a $C^2$ function.

Then, we say that $u$ is a harmonic function of order $s$ if $u$ satisfies the following equation:
\[
  \Delta^s u = 0.
\]
\end{Def}

Following \propname gives the expansion of harmonic functions of order $s$.
For the proof, see \cite{polyharm}.

\begin{prop}[Almansi Expansion] \label{almansi}
Let us $\Omega \subset \RR^n$ be a star domain and $u\colon \Omega \to \RR$ be a
  harmonic function of order $s$.

Then, there exists $s$ polyharmonic functions $h_0, h_1, \ldots, h_{s-1}$ on $\Omega$
such that
\[
  u = \sum_{j=0}^{s-1} r^{2s}h_j.
\]
\end{prop}

Let $P_k$ be the set of all homogeneous polynomials(not necessarily harmonic polynomials) of degree $k$ of two variables.

\begin{prop} \label{phharm}
See \cite{pshk}.
Let $k$ be a natural number more than or equal to 2.
Then
\[
  P_k = \Harm_k(\RR^2) \oplus r^2P_{k-2}.
\]
\end{prop}

We consider polyharmonic polynomials of order $s$.

Let us analyze $P_s\Harm_k(\RR^2) = \set{p_1f_k + p_2g_k; p_1, p_2 \in P_s}$.

\begin{Def}
Let $k$ and $s$ be two natural numbers such that $k \ge 2s$.
Then, we define $\Delta_k^s\colon P_k \to P_{k-2s}$ as
\begin{align*}
  \Delta_k^s = \underbrace{\Delta \Delta \cdots \Delta}_{s}.
\end{align*}
\end{Def}

\begin{prop} \label{pshk1}
Let $k$ and $s$ be two natural numbers such that $k \ge 2s$. Then
\[
  \Ker \Delta_k^s \supset P_{s-1}\Harm_{k-s+1}(\RR^2).
\]
\end{prop}
\begin{proof}
We will prove \propname \ \ref{pshk1} by induction on $s$.

If $s=1$, this Proposition is clearly.

Assume that this Proposition is correct for $s=1, 2, \ldots, s_0$.
We want to prove that
\[
  \Ker \Delta_k^{s_0+1} \supset P_{s_0}\Harm_{k-s_0}(\RR^2)
\]
i.e.
\begin{align*}
& x^{s_0-j}y^jf_{k-s_0}, x^{s_0-j}y^jg_{k-s_0} \in \Ker \Delta_k^{s_0+1}
  (\forall j \in \set{0, 1, 2, \ldots, s_0}) \\
& \same \Delta(x^{s_0-j}y^j f_{k-s_0}), \Delta(x^{s_0-j}y^j g_{k-s_0})
  \in \Ker \Delta_{k-2}^{s_0} \supset P_{s_0-1} \Harm_{k-s_0-1}(\RR^2)
  (\forall j \in \set{0, 1, 2, \ldots, s_0}).
\end{align*}

By direct computing, we have
\begin{align*}
\Delta(x^{s_0-j}y^jf_{k-s_0})
  &= 2(k-s_0)(s_0-j)x^{s_0-j-1}y^j g_{k-s_0-1} - 2j(k-s_0)x^{s_0-j}y^{j-1}f_{k-s_0-1} \\
  &\quad + 2(s_0-j)(s_0-j-1)x^{s_0-j-2}y^jf_{k-s_0} + j(j-1)x^{s_0-j}y^{j-2}f_{k-s_0}.
\end{align*}
Furthermore
\begin{align*}
x^{s_0-j-2}y^jf_{k-s_0}
  &= x^{s_0-j-2}y^j(xf_{k-s_0-1} - yg_{k-s_0-1}) \\
  &= x^{s_0-j-1}y^jf_{k-s_0-1} - x^{s_0-j-2}y^{j+1}g_{k-s_0-1} \in P_{s_0-1}\Harm_{k-s_0-1}(\RR^2).
\end{align*}
Similarly,
\[
  x^{s_0-j}y^{j-2}f_{k-s_0} \in P_{s_0-1}\Harm_{k-s_0-1}(\RR^2).
\]
Thus,
\[
  \Delta(x^{s_0-j}y^jf_{k-s_0}) \in P_{s_0-1}\Harm_{k-s_0-1}(\RR^2).
\]
Similarly, we obtain
\[
  \Delta(x^{s_0-j}y^j g_{k-s_0}) \in P_{s_0-1}\Harm_{k-s_0-1}(\RR^2).
\]
\end{proof}

\begin{prop} \label{pshk}
Let $k$ be a natural number and $s$ be a integer more than or equal to 0.
Then
\[
P_s\Harm_k(\RR^2)
  = \begin{cases}
    \Ker \Delta_{s+k}^{s+1} & (s < k-1) \\
    P_{s+k} & (s \ge k-1)\\
  \end{cases}.
\]
\end{prop}
\begin{proof}
First, we will prove the case of $s < k-1$.
By \propname \ \ref{almansi} and \propname \ \ref{pshk1}, we obtain that
\[
P_s\Harm_k(\RR^2) \subset \Ker \Delta_{s+k}^{s+1}
  = \bigoplus_{j=0}^s r^{2j}\Harm_{s+k-2j}(\RR^2).
\]
We want to prove that $P_s\Harm_k(\RR^2) \supset \Ker\Delta_{s+k}^{s+1}
= \bigoplus_{j=0}^s r^{2j}\Harm_{s+k-2j}(\RR^2)$.

For all $j \in \set{0, 1, 2, \ldots, s}$, we have $P_j\Harm_{s+k-j}(\RR^2) \subset P_s\Harm_k(\RR^2)$.
Also, by \propname \ \ref{harm-prop2},
\begin{align*}
r^{2j}f_{s+k-2j} &= f_jf_{s+k-j} + g_jg_{s+k-j}, \\
r^{2j}g_{s+k-2j} &= -g_jf_{s+k-j} + f_jg_{s+k-j}.
\end{align*}
Thus, $r^{2j}\Harm_{s+k-2j}(\RR^2) \subset P_j\Harm_{s+k-j}(\RR^2) \subset P_s\Harm_k(\RR^2)$.

Next, we will prove the case of $s=k-1$.
By the first statement of \propname \ \ref{pshk},
\begin{align*}
P_{k-1}\Harm_k(\RR^2)
  &\supset P_{k-2}\Harm_{k+1}(\RR^2) \\
  &= \Ker \Delta_{2k-1}^{k-1} \\
  &= \Harm_{2k-1}(\RR^2) \oplus r^2\Harm_{2k-3}(\RR^2) \oplus
    \cdots \oplus r^{2(k-2)}\Harm_3(\RR^2).
\end{align*}
Also, by \propname \ \ref{phharm}, we have
\[
  P_{2k-1} = \Harm_{2k-1}(\RR^2) \oplus r^2\Harm_{2k-3}(\RR^2) \oplus \cdots
    \oplus r^{2(k-2)}\Harm_3(\RR^2) \oplus r^{2(k-1)}\Harm_1(\RR^2).
\]
Furthermore, by \propname \ \ref{harm-prop2}, we have
\begin{align*}
  xr^{2(k-1)} &= f_{k-1}f_k + g_{k-1}g_k, \\
  yr^{2(k-1)} &= -g_{k-1}f_k + f_{k-1}g_k.
\end{align*}

We will prove in the case of $s > k-1$.

If $s-k+1$ is even,
\begin{align*}
  P_s\Harm_k(\RR^2)
    &\supset P_{s-\frac{s-k+1}{2}}\Harm_{k+\frac{s-k+1}{2}} \\
    &= P_{\frac{s+k-1}{2}}\Harm_{\frac{s+k+1}{2}} = P_{s+k}.
\end{align*}

If $s-k+1$ is odd,
\begin{align*}
  P_s\Harm_k(\RR^2)
    &\supset P_{s-\frac{s-k}{2}} \Harm_{k + \frac{s-k}{2}}(\RR^2) \\
    &= P_{\frac{s+k}{2}}\Harm_{\frac{s+k}{2}}.
\end{align*}
Let $l = \frac{s+k}{2}$. We want to prove that $P_l\Harm_l = P_{2l}$.
However, we have
\begin{align*}
P_l\Harm_l(\RR^2)
  &\supset P_{l-1}\Harm_{l+1}(\RR^2) \\
  &= \Ker \Delta_{2l}^l \\
  &= \Harm_{2l}(\RR^2) \oplus r^2\Harm_{2l-2}(\RR^2) \oplus \cdots \oplus r^{2(l-1)}\Harm_2(\RR^2)
\end{align*}
and $r^{2l} = f_lf_l + g_lg_l \in P_l\Harm_l$.
\end{proof}
%}}}

To conclude the section on harmonic functions, we classify all non-zero harmonic
functions:

\begin{prop} \label{classify-harm} %{{{
Let $h$ be a non-zero harmonic function of order $k$.
Then, $h$ is right equivalent to $f_k$.
\end{prop}
\begin{proof}
Since $h$ is a harmonic function, there exists a complex valued function $u\colon (\CC, 0) \to
(\CC, 0)$ such that
\[
  \Re \, u = h.
\]

Since order of $u$ is equal to $k$, there exists a complex valued function $\tilde{u}\colon (\CC, 0) \to
\CC$ such that $u = z^k\tilde{u}$ and $\tilde{u}(0) \neq 0$.

We define a local diffeomorphism $\phi\colon (\RR^2, \zero) \sim (\CC, 0) \to (\CC, 0) \sim
  (\RR^2, \zero)$ as
\[
  \phi(z) = z\sqrt[k]{\tilde{u}(z)}.
\]
Furthermore, we define $u_0\colon (\RR^2, \zero) \sim (\CC, 0) \to (\CC, 0) \sim (\RR^2, \zero)$
as
\[
  u_0(z) = z^k.
\]
Then, we obtain
\[
  u_0 \circ \phi = u.
\]
Thus, by taking real part of above equation, we obtain
\[
    f_k \circ \phi = h.
\]
\end{proof}
%}}}

% }}}

\section{Main Results And Proofs} % {{{
\begin{thm} \label{main} % {{{
Let $k$ be a natural number more than or equal to 5 and let $h_k$ be a
  non-zero homogeneous harmonic polynomial with degree $k$.
Furthermore, for all $s \in \set{1, 2, \ldots, k-4}$, we take $\rho_{k+s} \in \Ker
  \Delta_{k+s}^{\sigma_s}$ where $\sigma_s$ is defined by
\[
  \sigma_s = \begin{cases}
    s + 1 & (s < \frac{k-3}{2}) \\
    s + 2 & (s \ge \frac{k-3}{2})
  \end{cases}.
\]
Then, for all $\rho_{k+s} \in \Ker \Delta_{k+s}^{\sigma_s}$ and $ R_{2k-3} \in \m_2^{2k-3}$,
  we have $h_k$ and $h_k + \sum_{s=1}^{k-4}\rho_{k+s} + R_{2k-3}$ are right equivalent.
\end{thm} % }}}

% proof of main {{{
For proving \theoname \ \ref{main}, we prepare some \lemname{s}.

\begin{Lem} \label{determined}
Let $k$ be a natural number.
Furthermore, let $h_k$ be a non-zero homogeneous harmonic polynomial with degree $k$
and $R_{k+1} \in \m_2^{k+1}$.
Then, $h_k + R_{k+1}$ is $\max (k, 2k-4)$-determined.
\end{Lem}
\begin{proof}
If $k \le 4$, we proved that $h_k$ is $k$-determined(see \propname \ \cite{mine}).
We suppose that $k \ge 5$.

We can take $h_k = f_k$ without loss of generality (see \ref{classify-harm}).

So,
\begin{align*}
J_{f_k + R_{k+1}}
  &= \langle kf_{k-1} + \frac{\del R_{k+1}}{\del x},
    -kg_{k-1} + \frac{\del R_{k+1}}{\del y} \rangle_{\E_2}. \\
\end{align*}
Furthermore, by \propname \ \ref{pshk}, we have
$\m_2^{2k-3} \subset \m_2 J_{f_k + R_{k+1}} + \m_2^{2k-2}$.
Thus, by \propname \ \ref{pshk}, we have $f_k + R_{k+1}$ is $(2k-3)$-determined.

Let $u, v$ be two elements in $P_{k-2}$.
Furthermore, we define $\phi\colon (\RR^2, \zero) \to (\RR^2, \zero)$ be a local diffeomorphism as
\[
  \phi(x, y) = (x + u, y + v).
\]
Then, we have
\[
  (f_k + R_{k+1}) \circ \phi
    \sim_{j^{2k-3}} f_k + R_{k+1} + k(uf_{k-1} - vg_{k-1}).
\]
Also, $\set{k(uf_{k-1}) - vg_{k-1}; u, v \in P_{k-2}} = P_{k-2}\Harm_{k-1}(\RR^2) = P_{2k-3}$.
Thus, we have the \propname.
\end{proof}

Following \lemname \ is used for $s < \frac{k-3}{2}$.

\begin{Lem} \label{lough-main1} % {{{
Let $k$ be a natural number more than or equal to 5 and let $h_k$ be a
  non-zero homogeneous harmonic polynomial with degree $k$.

Furthermore, for all $s \in \set{1, 2, \ldots, k-4}$, we take $\rho_{k+s} \in \Ker
  \Delta_{k+s}^{s + 1}$.

Then, $h_k$ and $h_k + \sum_{s=1}^{k-4}\rho_{k+s}$ are right equivalent.
\end{Lem}
\begin{proof}
We can take $h_k = f_k$ without loss of generality (see \propname \ \ref{classify-harm}).

Furthermore, by \propname \ \ref{pshk}, for all $s \in \set{1, 2, \ldots, k-4}$, $\rho_{k+s} \in \Ker \Delta_{k+s}^{s+1}
  = P_s\Harm_k(\RR^2)$, there exists $u, v \in \m_2$ such that $\sum_{s=1}^{k-4}\rho_{k+s}
    = uf_k + vg_k$.

Then, we identify $(\RR^2, \zero)$ and $(\CC, 0)$ and we define a local diffeomorphism $\phi\colon (\RR^2, \zero) \to
  (\RR^2, \zero)$ as
\[
  \phi(z) = z\sqrt[k]{1 + u -iv}.
\]
Then, we have $f_k \circ \phi = f_k + uf_k + vg_k$.
\end{proof} % }}}

Following \lemname \ is used for $s \ge \frac{k-3}{2}$.

\begin{Lem} \label{high_deg-main} % {{{
Let $k$ be a natural number more than or equal to 5 and let $h_k$ be a
  non-zero homogeneous harmonic polynomial with degree $k$.

For all $s \in \set{1, 2, \ldots, k-4}$, we take $\rho_{k+s} \in \Ker \Delta_{k+2}^{s+2}$.
Then, there exists $u, v \in P_{s+1}$ such that when we define a local diffeomorphism $\phi\colon
  (\RR^2, \zero) \to (\RR^2, \zero)$ as $\phi(x, y) = (x + u, y + v)$, then
  $h_k \circ \phi \sim_{j^{k+s}} h_k + \rho_{k+s}$.
\end{Lem}
\begin{proof}
We can take $h_k = f_k$ without loss of generality (see \propname \ref{classify-harm}).

Then,
\begin{align*}
f_k \circ \phi
  &= \Re (x + iy + u + iv)^k \\
  &\sim_{j^{k+s}} \Re (x + iy)^k + k \Re(x + iy)^{k-1} (u + iv) \\
  &= f_k + k(uf_{k-1} - vg_{k-1}).
\end{align*}
Since \propname \ \ref{pshk}, $\set{k(uf_{k-1} - vg_{k-1}); u, v \in P_s} = P_{s+1}\Harm_{k-1}
  = \Ker \Delta_{s+k}^{s+2}$.

Thus, we obtain the statement.
\end{proof} % }}}

\begin{proof}[Proof of \theoname \ \ref{main}]
By \lemname \ \ref{determined}, we can take $R_{2k-3} = 0$
  without loss of generality.

By \lemname \ \ref{high_deg-main}, for all natural number $s \ge \frac{k-3}{2}$,
  there exist $u_{s+1}, v_{s+1} \in P_{s+1}$ such that if we define a local diffeomorphism
  $\phi_{s+1}\colon (\RR^2, \zero) \to (\RR^2, \zero)$ as $\phi_{s+1}(x, y) = (x + u_{s+1}, y +
    v_{s+1})$, then $h_k \circ \phi_{s+1} \sim_{j^{2k-4}} h_k - \rho_{k+s}$.
Let $s_0$ be a minimum natural number such that $s \ge \frac{k-3}{2}$.
Then, we have
\[
  (h_k + \sum_{s=1}^{k-4}\rho_{k+s}) \circ \phi_{s_0 + 1} \circ \phi_{s_0+2} \circ \cdots \circ
    \phi_{k-3} \sim_{j^{2k-4}} h_k + \sum_{s=1}^{s_0-1}\rho_{k+s}.
\]

Finally, by \lemname \ \ref{lough-main1}, $h_k + \sum_{s=1}^{s_0-1} \rho_{k+s}$ and $h_k$ are right
equivalent.
\end{proof}
% }}}
% }}}

% "{{{ bibliography

% "}}}
\end{document}